\documentclass[12pt,leqno]{amsart}
\usepackage{color,amsmath,amsfonts,amssymb,amscd,amsthm,amsbsy,upref}
\usepackage{tikz}
\usepackage{tkz-euclide}
\usepackage{bbm}
\usepackage{multicol}
\usepackage{enumitem}
\usepackage[utf8]{inputenc}
\usepackage[T1]{fontenc}
\usepackage{graphicx,bigints}

\numberwithin{equation}{section}
\def\hangbox to #1 #2{\vskip3pt\hangindent #1\noindent \hbox to #1{#2}$\!\!$}

\newtheorem{thm}{Theorem}[section]

\newtheorem{lem}[thm]{Lemma}
\newtheorem{cor}[thm]{Corollary}

\newtheorem{prop}[thm]{Proposition}

\theoremstyle{definition}

\theoremstyle{remark}

\def\C{{\mathbb C}}
\def\N{{\mathbb N}}
\def\R{{\mathbb R}}

\def\sfrac#1#2{\kern.1em\raise.5ex\hbox{$#1$}
        \kern-.1em/\kern-.05em\lower.25ex\hbox{$#2$}}

\def\vp{\varepsilon}

\newcommand{\fw}{\text{\fw}}

\begin{document}

\title{Locality and stability for phase retrieval}

\author{Wedad Alharbi}\author{Salah Alshabhi}\author{ Daniel Freeman}\author{ Dorsa Ghoreishi}
\address{Department of Mathematics and Statistics\\
St Louis University\\
St Louis MO 63103  USA
} \email{wedad.alharbi@slu.edu}\email{salah.alshabhi@slu.edu}
\email{daniel.freeman@slu.edu}
\email{dorsa.ghoreishi@slu.edu}

\begin{abstract}
    A frame $(x_j)_{j\in J}$ for a Hilbert space $H$ is said to do phase retrieval if for all distinct vectors $x,y\in H$  the magnitude of the frame coefficients $(|\langle x, x_j\rangle|)_{j\in J}$ and $(|\langle y, x_j\rangle|)_{j\in J}$ distinguish $x$ from $y$ (up to a unimodular scalar).  We consider the weaker condition where the magnitude of the frame coefficients distinguishes $x$ from every vector $y$  in a small neighborhood of $x$ (up to a unimodular scalar).  We prove that some of the important theorems for phase retrieval hold for this local condition, where as some theorems are completely different.  We prove as well that when considering stability of phase retrieval, the worst stability inequality is always witnessed at orthogonal vectors.  This allows for much simpler calculations when considering optimization problems for phase retrieval.
    
\end{abstract}


\thanks{The third author was supported by grant 706481 from the Simons Foundation.  The third and fourth authors were supported by grant 2154931 from the National Science Foundation.}

\maketitle

\section{Introduction}

We say that a collection of vectors  $(x_j)_{j\in J}\subseteq H$ is a {\em frame} of a Hilbert space $H$ if there exists constants $B\geq A>0$ so that
\begin{equation}
    A\|x\|^2\leq\sum_{j\in J}|\langle x,x_j\rangle|^2\leq B\|x\|^2\hspace{1cm}\textrm{ for all }x\in H.
\end{equation}
The {\em analysis operator} of the frame $(x_j)_{j\in J}$ is the map $\Theta:H\rightarrow\ell_2(J)$ given by $\Theta(x)=(\langle x,x_j\rangle)_{j\in J}$ for all $x\in H$.\\  

Instead of considering a discrete collection of vectors, we now let $(x_t)_{t\in \Omega}\subseteq H$ be indexed by a measure space $(\Omega,\mu)$.  We say that  $(x_t)_{t\in \Omega}\subseteq H$ is a {\em continuous frame} of $H$ if there exists constants $B\geq A>0$ so that
\begin{equation}
    A\|x\|^2\leq\int_{t\in \Omega}|\langle x,x_t\rangle|^2\leq B\|x\|^2\hspace{1cm}\textrm{ for all }x\in H.
\end{equation}
The {\em analysis operator} of the continuous frame $(x_t)_{t\in \Omega}$ is the map $\Theta:H\rightarrow L_2(\Omega)$ given by $\Theta(x)=(\langle x,x_j\rangle)_{t\in \Omega}$ for all $x\in H$.  Note that every frame $(x_j)_{j\in J}$ may be realized as a continuous frame by endowing $J$ with counting measure.\\  

A frame $(x_j)_{j\in J}$ for a Hilbert space $H$  allows for a continuous, linear, and stable reconstruction of any vector $x\in H$ from the collection of frame coefficients $\Theta(x)=(\langle x,x_j\rangle)_{j\in J}$. 
However, there are many situations where one is only able to obtain the magnitude of the frame coefficients $|\Theta(x)|=(|\langle x,x_j\rangle|)_{j\in J}$.  This problem arises in speech recognition \cite{BR}, coherent diffraction imaging \cite{MCKS}, X-ray crystallography \cite{T}, transmission electron microscopy \cite{K}, and in many other areas of physics, engineering, signal processing, and applied mathematics.  The theory of phase retrieval was developed to tackle this problem, and it has grown into an exciting area of mathematical research.   We recommend \cite{GKR} and \cite{FaS} for surveys on the mathematics of phase retrieval.\\

Let $(x_t)_{t\in \Omega}$ be a continuous frame of Hilbert space $H$ with analysis operator $\Theta:H\rightarrow L_2(\Omega)$.  If $x,y\in H$ and $x=\lambda y$ for some  scalar $|\lambda|=1$ then $|\Theta(x)|=|\Theta(y)|$.  Thus, although it is possible to recover any vector $x\in H$ from the frame coefficients $\Theta(x)=(\langle x,x_t\rangle)_{t\in\Omega}$, it is not possible to distinguish $x$ and $\lambda x$ using only the magnitudes of the frame coefficients.  We say that $(x_t)_{t\in \Omega}$ does {\em phase retrieval} if whenever $x,y\in H$ and $|\Theta(x)|=|\Theta(y)|$ we have that $x=\lambda y$ for some  scalar $|\lambda|=1$.  In applications it is not a problem if one recovers a signal $\lambda x$ instead of $x$.  For example, in acoustics, the sound wave $x$ sounds exactly the same as the sound wave $\lambda x$, and in X-ray crystallography, multiplying $x$ by $\lambda$ corresponds to a rigid rotation and preserves the structure of the crystal.  As any application of phase retrieval will involve error, it is important that we are not only able to obtain $x$ (up to a unimodular scalar $\lambda$) from the magnitudes of the frame coefficients $|\Theta x|$, but that this recovery be also stable.  That is, we want that the recovery of $[x]_\sim$ from $|\Theta x|$ to be Lipschitz continuous, where $\sim$ is the equivalence relation on $H$ given by $x\sim y$ if and only if $x=\lambda y$ for some $|\lambda|=1$.\\  

Let $(x_t)_{t\in \Omega}$ be a continuous frame of a Hilbert space $H$ with analysis operator $\Theta:H\rightarrow L_2(\Omega)$ and let $C>0$.  We say that  $(x_t)_{t\in \Omega}$ does {\em $C$-stable phase retrieval} if 
\begin{equation}\label{E:intro}
    \min_{|\lambda|=1}\|x-\lambda y\|\leq C\big\||\Theta x|-|\Theta y|\big\|\hspace{1cm}\textrm{ for all }x,y\in H.
\end{equation}

Calculating stability bounds for phase retrieval can be notoriously difficult and requires checking \eqref{E:intro} for all pairs of vectors $x,y\in H$.  Our first contribution in Section \ref{S:ortho} is to prove the following theorem which implies that $(x_t)_{t\in \Omega}$ does $C$-stable phase retrieval if and only if \eqref{E:intro} holds for all pairs of orthogonal vectors $x,y\in H$.  This greatly simplifies many calculations as if $x$ and $y$ are orthogonal then $\|x-\lambda y\|=(\|x\|^2+\|y\|^2)^{1/2}$.

\begin{thm}\label{T:Iortho}
Let $(x_t)_{t\in\Omega}$ be a continuous frame for a Hilbert space $H$ with analysis operator $\Theta:H\rightarrow L_2(\Omega)$.  Let $x,y\in H$. Then there exists $x_o,y_o\in span\{x,y\}$ with $\|x_o\|=1$, $\|y_o\|\leq 1$ and $\langle x_o,y_o\rangle=0$ such that
$$\frac{\big\||\Theta x|-|\Theta y|\big\|}{\min_{|\lambda|=1}\|x-\lambda y\|}\geq \frac{\big\||\Theta x_o|-|\Theta y_o|\big\|}{\min_{|\lambda|=1}\|x_o-\lambda y_o\|}=\frac{\big\||\Theta x_o|-|\Theta y_o|\big\|}{(1+\|y_o\| |^2)^{1/2}}.
$$
\end{thm}

Theorem \ref{T:Iortho} implies that the worst stability bounds for phase retrieval occur at orthogonal vectors.  That is, the minimum of the ratio $\big\||\Theta x|-|\Theta y|\big\|/\min_{|\lambda|=1}\|x-\lambda y\|$ occurs at  orthogonal vectors.  We now fix $x\in H$ and consider the ratio $\big\||\Theta x|-|\Theta y|\big\|/\min_{|\lambda|=1}\|x-\lambda y\|$ only for values of $y$ in a small neighborhood of $x$.  Let $(x_t)_{t\in\Omega}$ be a continuous frame for a Hilbert space $H$ with analysis operator $\Theta:H\rightarrow L_2(\Omega)$.  We say that a continuous frame does {\em $C$-stable phase retrieval near  $x\in H$} if 
\begin{equation}
   1\leq \lim_{y\rightarrow x}C\frac{\||\Theta x|-|\Theta y|\|}{\min_{|\lambda|=1}\|x-\lambda y\|}.
\end{equation}

There are many instances where it is proven that a particular frame or continuous frame fails to do stable phase retrieval by proving that the stability bound is arbitrarily bad at some orthogonal vectors. 
However, there are important instances where stable phase retrieval can be achieved by restricting the vectors $x,y\in H$ to a subset with some specified structure \cite{ADGY},\cite{CCSW}\cite{CDDL},\cite{GR}.
We consider the stability of phase retrieval when $x\in H$ is fixed and $y$ is restricted to a small neighborhood of $x$.  In Section \ref{S:R} we prove that if $(x_t)_{t\in\Omega}$ is any continuous frame of $\R^n$ and $x\in\R^n$ then 
$(x_t)_{t\in\Omega}$ does stable phase retrieval near $x$.  In particular, if $x\in\R^n$ and $(e_j)_{j=1}^n$ is a basis of $\R^n$ then $(e_j)_{j=1}^n$ does stable phase retrieval near $x$.  This is significant as if $(x_j)_{j=1}^m$ is a frame which does phase retrieval for $\R^n$ then it is necessary that $m\geq 2n-1$ \cite{BCE}.  In Section \ref{S:C} we show that this local stability fails in $\C^n$.  That is, we prove that if $(e_j)_{j=1}^n$ is a basis of $\C^n$ and $x\in\C^n$ is not a multiple of some $e_j$ then  $(e_j)_{j=1}^n$ does not do stable phase retrieval near $x$. \\

In Section \ref{S:inf} we consider stable phase retrieval for infinite dimensional Hilbert spaces.  It is known that no frame or continuous frame for an infinite dimensional Hilbert space can do stable phase retrieval \cite{AG}\cite{CCD}.  In both papers, the instability is shown to occur at orthogonal vectors.  We prove that it is always the case that this instability can be witnessed locally as well.  In particular, we prove that if $(x_t)_{t\in\Omega}$ is a continuous frame of an infinite dimensional Hilbert space $H$ then the set of vectors $x\in H$ where $(x_t)_{t\in\Omega}$ fails to do stable phase retrieval near $x$ is dense in $H$.\\

Instead of considering the stability of phase retrieval near a vector $x$, one could consider norm retrieval or phase retrieval by projections \cite{CCJW}\cite{CGJT}.  We leave the corresponding statements for these topics as open problems.

\section{Orthogonality and stability}\label{S:ortho}
Recall that a continuous frame $(x_t)_{t\in\Omega}$ of $H$ with analysis operator $\Theta$ does $C$ stable phase retrieval if 
\begin{equation}\label{E:stab}
    \min_{|\lambda|=1}\|x-\lambda y\|\leq C\big\||\Theta x|-|\Theta y|\big\|\hspace{1cm}\textrm{ for all }x,y\in H.
\end{equation}

Proving that a frame or continuous frame does $C$-stable phase retrieval requires checking that \eqref{E:stab} is satisfied for every pair of vectors $x,y\in H$.
The following theorem implies that the worst stability constant for phase retrieval is witnessed at orthogonal vectors.  That is, a continuous frame does $C$-stable phase retrieval if and only if \eqref{E:stab} is satisfied for orthogonal vectors.   
\begin{thm}\label{T:ortho}
Let $(x_t)_{t\in\Omega}$ be a continuous frame for a Hilbert space $H$ with analysis operator $\Theta:H\rightarrow L_2(\Omega)$.  Let $x,y\in H$. Then there exists $x_o,y_o\in span(x,y)$ with $\|x_o\|=1$, $\|y_o\|\leq 1$ and $\langle x_o,y_o\rangle=0$ such that
$$\frac{\big\||\Theta x|-|\Theta y|\big\|}{\min_{|\lambda|=1}\|x-\lambda y\|}\geq \frac{\big\||\Theta x_o|-|\Theta y_o|\big\|}{\min_{|\lambda|=1}\|x_o-\lambda y_o\|}=\frac{\big\||\Theta x_o|-|\Theta y_o|\big\|}{(1+\|y_o\| |^2)^{1/2}}.
$$

\end{thm}

\begin{proof}

Let $x,y\in H$.  After multiplying $y$ by a unimodular scalar, we may assume that $\|x-y\|= \min_{|\lambda|=1}\|x-\lambda y\|$. Let $\lambda\in \C$ with $|\lambda|=1$.  We have that 
\begin{align*}
\|x-\lambda y\|^2&=\langle x-\lambda y,x-\lambda y\rangle\\
&=\langle x,x\rangle^2-\langle \lambda y,x\rangle-\langle x,\lambda y\rangle +\langle y,y\rangle^2\\
&=\langle x,x\rangle^2-2Real(\lambda\langle y,x\rangle) +\langle y,y\rangle^2
\end{align*}
Thus, we have that $\|x-y\|= \min_{|\lambda|=1}||x-\lambda y||$ is equivalent to $\langle x,y\rangle\in\R$ and $\langle x,y\rangle\geq0$.
As $\langle x-0(x+y),y-0(x+y)\rangle\geq 0$ and $\langle x-\frac{1}{2}(x+y),y-\frac{1}{2}(x+y)\rangle\leq 0$, there exists $R\in[0,1/2]$ such that $\langle x-R(x+y),y-R(x+y)\rangle=0$.  
We have for all $r\in [0,R]$ that $\langle x-r(x+y),y-r(x+y)\rangle\geq0$ and hence 
\begin{equation}\label{E:same}
\|x-y\|=\|x-r(x+y)-(y-r(x+y))\|=\min_{|\lambda|=1}\|x-r(x+y)-\lambda(y-r(x+y))\|
\end{equation}

Let $t\in\Omega$.  Consider the function $f_t:[0,1/2]\rightarrow \R$ given by
$$f_t(r)=\big|| \langle x-r(x+y),x_t\rangle|-|\langle y-r(x+y),x_t \rangle|\big|\hspace{1cm}\textrm{ for all }r\in[0,1/2].
$$
We will prove that $f_t$ is decreasing on $[0,1/2]$.  To simplify the notation, we let 
$\langle x,x_t \rangle=a+bi$ and  $\langle y,x_t\rangle=c+di$.  Note that $f_t$ is unchanged if we switch the roles of $x$ and $y$ or multiply $x$ and $y$ by the same unimodular scalar.  Thus, we may assume without loss of generality that $b=-d$ and that $a\geq 0$ and $a\geq |c|$.  This gives that 
\begin{align*}
f_t(r)&=\big||a+bi-r(a+c)|-|c-bi-r(a+c)|\big|\\
&=\big|((a-r(a+c))^2+b^2)^{1/2}-((c-r(a+c))^2+b^2)^{1/2}\big|\\
&=(a-r(a+c))^2+b^2)^{1/2}-((c-r(a+c))^2+b^2)^{1/2}\hspace{.5cm}\textrm{ as $a\geq|c|$ and $0\leq r\leq 1/2$}.\\
\end{align*}
We first consider the case where $b=0$ and $r(a+c)\leq c< a$. Then,
$$f_t(r)=|a-r(a+c)|-|c-r(a+c)|=(a-r(a+c))-(c-r(a+c)=a-c=f(0).$$
Thus, $f_t$ is constant for those values of $r$.  We now consider the case where $b=0$ and $c<r(a+c)< a$.  Then,
 $$f_t(r)=|a-r(a+c)|-|c-r(a+c)|=a-r(a+c)+(c-r(a+c))=(a+c)(1-2r).$$
Thus $f_t$ is decreasing as $a\geq|c|$.  This proves that $f_t$ is a decreasing function when $b=0$.  We now assume that $b\neq 0$.  Note that $f_t(0)\geq0$ and $f_t(1/2)=0$.  In this case, one can check that $f_t$ is a differentiable function and 
$f_t$ has no critical point in the interval $(0,1/2)$.  Thus, $f_t$ is decreasing on $[0,1/2]$.  As $f_t$ is decreasing, we have that 
$f_t(0)\geq f_t(R)$.  Hence,
\begin{equation}\label{E:dec1}
\big||\langle x,x_t\rangle|-|\langle y,x_t\rangle|\big|\geq \big||\langle x-R(x+y),x_t\rangle|-|\langle y-R(x+y),x_t\rangle|\big|\hspace{.5cm}\textrm{ for all }t\in\Omega.
\end{equation}
Thus, we have that 
\begin{equation}\label{E:dec2}
 \big\||\Theta x|-|\Theta y|\big\|\geq\big\||\Theta (x-R(x+y))|-|\Theta (y-R(x+y))|\big\|
\end{equation}
By \eqref{E:same} and \eqref{E:dec2} we have that
\begin{equation}\label{E:dec}
\frac{\big\||\Theta x|-|\Theta y|\big\|}{\min_{|\lambda|=1}\|x-\lambda y\|}\geq \frac{\big\||\Theta (x-R(x+y))|-|\Theta (y-R(x+y))|\big\|}{\min_{|\lambda|=1}\|(x-R(x+y))-\lambda (y-R(x+y))\|}
\end{equation}
Without loss of generality, we may assume that $\|x-R(x+y)\|\geq \|x-R(x+y)\|$.  We let $x_o=(x-R(x+y))/\|x-R(x+y)\|$ and $y_o=(y-R(x+y))/\|x-R(x+y)\|$.  Then $x_o$ and $y_o$ are orthogonal, $\|x_o\|=1$, and $\|y_o\|\leq 1$.  Furthermore,  \eqref{E:dec} implies that
$$\frac{\big\||\Theta x|-|\Theta y|\big\|}{\min_{|\lambda|=1}\|x-\lambda y\|}\geq \frac{\big\||\Theta x_o|-|\Theta y_o|\big\|}{\min_{|\lambda|=1}\|x_o-\lambda y_o\|}=\frac{\big\||\Theta x_o|-|\Theta y_o|\big\|}{(1+\|y_o\| |^2)^{1/2}}.
$$
\end{proof}

We now discuss some consequences of Theorem \ref{T:ortho}.
  We define a function $\Psi$ on a subset of $H\times H$ by 
\begin{equation}\label{E:psi}
    \Psi(x,y)=\frac{\||\Theta x|-|\Theta y|\|}{\min_{|\lambda|=1}\|x-\lambda y\|}\hspace{1cm}\textrm{ for all $x,y\in H$ with $x\neq \lambda y$ for all $|\lambda|=1$.}
\end{equation}
The optimal value $C$ for which $(x_t)_{t\in\Omega}$ does $C$-stable phase retrieval is then given by $C=(\inf \Psi(x,y))^{-1}$.  The problem with this approach is that  the infimum is being taken over a non-compact set. 
However, by Theorem \ref{T:ortho} we only need to consider $\Psi(x,y)$ for orthogonal vectors $x,y\in H$.  This gives the following immediate corollary.

\begin{cor}\label{C:2.2}
Let $(x_t)_{t\in\Omega}$ be a continuous frame for a Hilbert space $H$ with analysis operator $\Theta:H\rightarrow L_2(\Omega)$.  Let $X\subseteq H\oplus H$ be defined by 
$$X=\{(x,y)\in H\times H: \|x\|=1,\|y\|\leq 1,\textrm{ and }\langle x,y\rangle =0.\}.
$$
Then the following all hold.
\begin{enumerate}
    \item $(x_t)_{t\in\Omega}$ does stable phase retrieval if and only if $\inf_{(x,y)\in X}\Psi(x,y)>0$. 
    \item If $(x_t)_{t\in\Omega}$ does stable phase retrieval then the smallest value $C$ so that $(x_t)_{t\in\Omega}$ does $C$-stable phase retrieval is given by $C=(\inf_{(x,y)\in X}\Psi(x,y))^{-1}$.
    \item The function $\Psi$ is continuous on $X$ and is given by,
        $$\Psi(x,y)=\frac{\||\Theta x|-|\Theta y|\|}{(1+\|y\|^2)^{1/2}}\hspace{1cm}\textrm{ for all $(x,y)\in X$.}
    $$
    
    \item If $H$ is finite dimensional then $X\subseteq H\oplus H$ is a compact set.
\end{enumerate}
\end{cor}

It is well known that  a continuous frame for a finite dimensional Hilbert space does phase retrieval if and only if it does $C$-stable phase retrieval for some constant $C$.  This fundamental result in the mathematics of phase retrieval has been proven in multiple papers using a variety of methods \cite{AG,BCMN,BW,CCD}.   We now apply Theorem \ref{T:ortho} and Corollary \ref{C:2.2} to give a simple and natural proof.

\begin{cor}
 Let $(x_t)_{t\in\Omega}$ be a continuous frame for a finite dimensional Hilbert space $H$.  Then $(x_t)_{t\in\Omega}$ does phase retrieval if and only if $(x_t)_{t\in\Omega}$ does $C$-stable phase retrieval for some constant $C$.
\end{cor}
\begin{proof}
Suppose that $(x_t)_{t\in\Omega}$ does phase retrieval on $H$. 
Note that if $(x,y)\in X$ then $x\neq \lambda y$ for all $|\lambda|=1$ as $\|x-\lambda y\|^2=1+\|y\|^2$.  Thus, 
 for all $(x,y)\in X$ we have that $|\Theta x|\neq |\Theta y|$ as $(x_t)_{t\in\Omega}$ does phase retrieval on $H$. Hence, $\Psi(x,y)>0$ for all $(x,y)\in X$.  As $\Psi$ is a continuous function on the compact set $X$ we have that $\min_{(x,y)\in X}\Psi(x,y)>0$.  Thus, $(x_t)_{t\in\Omega}$ does $C$-stable phase retrieval for $C=(\min_{(x,y)\in X}\Psi(x,y))^{-1}$.
\end{proof}

\section{Stable phase retrieval near $x\in\R^n$}\label{S:R}

Let $(x_j)_{j\in J}$ be a frame of $\R^n$ with analysis operator $\Theta:\R^n\rightarrow \ell_2(J)$.  By Theorem \ref{T:ortho}, we have that if $(x_j)_{j\in J}$ does phase retrieval then the worst stability bounds occur at orthogonal vectors.  
Furthermore, it was known that if $(x_j)_{j\in J}$ does not do phase retrieval then there are orthogonal vectors $x,y\in\R^n$ such that $|\Theta x|=|\Theta y|$ \cite{BCE}.
We now show that very good stability bounds for phase retrieval can be obtained if we fix $x\in\R^n$ and only consider $y$ in a small neighborhood of $x$.  In particular, even if $(x_j)_{j\in J}$  fails to do phase retrieval for $\R^n$, we still have that $(x_j)_{j\in J}$ does phase retrieval near $x\in\R^n$.

\begin{prop}\label{P:frame}
Let $n\in\N$ and let $(x_j)_{j\in J}$ be a finite frame of $\R^n$ with lower frame bound $A$ and analysis operator $\Theta:\R^n\rightarrow \ell_2(\Theta)$.  Let $x\in\R^n$ and $J_x=\{j\in J:\langle x,x_j\rangle\neq 0\}$.  Then for all $y\in\R^n$ with $\|x-y\|\leq \min_{j\in J_x}|\langle x,x_j/\|x_j\|\rangle|$ we have that $\min_{|\lambda|=1}\|x-\lambda y\|\leq A^{-1/2}\| |\Theta(x)|-|\Theta(y)|\|$.
In particular, $(x_j)_{j\in J}$ does $A^{-1/2}$-stable phase retrieval near $x$.
\end{prop}

\begin{proof}
Let $x \in R^n$ with $x \neq 0$ and $\beta = \min_{j\in J_x}|\langle x,x_j/\|x_j\|\rangle| $.  Let $y \in R^n$ with $\|x-y\| < \beta $. We first show that $sign(\langle x,x_j\rangle) =sign(\langle y,x_j\rangle) $ for all $j\in J_x$. Let $j\in J_x$ and without loss of generality assume that $\langle x,x_j\rangle >0$. We have that
$$\langle y,x_j\rangle= \langle x,x_j\rangle +\langle y-x,x_j\rangle  \geq \langle x,x_j\rangle -\|x-y\|\|x_j\|  >  \langle x,x_j\rangle - \frac{| \langle x,x_j\rangle|}{\|x_j\|}\|x_j\| = 0.$$ 
Thus, $sign(\langle x,x_j\rangle) =sign(\langle y,x_j\rangle) $ for $j\in J_x$.
We now have that

\begin{align*}
  \big\| &|\Theta(x)|-|\Theta(y)|\big\|^2_{l_2(J)}= \sum_{j\in J}\big||\langle x,x_j\rangle|-|\langle y,x_j\rangle|\big|^2\\
  &=\sum_{j\in J_x}\big||\langle x,x_j\rangle|-|\langle y,x_j\rangle|\big|^2 + \sum_{j \notin J_x}|\langle y,x_j\rangle|^2 \hspace{.5cm}\textrm{ as $\langle x,x_j\rangle=0$ for $j\not\in J_x$},\\
  &=\sum_{j\in J_x}|\langle x,x_j\rangle-\langle y,x_j\rangle|^2 + \sum_{j \notin J_x}|\langle y,x_j\rangle|^2\hspace{.5cm}\textrm{ as }\textrm{$sign(\langle x,x_i\rangle) =sign(\langle y,x_j\rangle) $ for $j\in J_x$,}\\
  &=\sum_{j\in J_x}|\langle x-y,x_j\rangle|^2 + \sum_{j \notin J_x}|\langle x-y,x_j\rangle|^2\hspace{.5cm}\textrm{ as $\langle x,x_j\rangle=0$ for $j\not\in J_x$},\\
  &=\sum_{j\in J}|\langle x-y,x_j\rangle|^2 \geq A \|x-y\|^2\hspace{.5cm}\textrm{ as $(x_j)_{j\in J}$ has lower frame bound $A$.}
\end{align*}

\end{proof}

Note that Proposition \ref{P:frame} implies in particular that every ortho-normal basis for $\R^n$ does $1$-stable phase retrieval near $x\in\R^n$ for every $x\in\R^n$.  This is notable in that for a frame to do phase retrieval for $\R^n$, the frame must have at least $2n-1$ vectors.  The following theorem extends Proposition \ref{P:frame} to continuous frames for $\R^n$.

\begin{thm}\label{T:cframe}
Let $n\in\N$ and let $(x_t)_{t\in\Omega}$ be a continuous frame for $\R^n$ with lower frame bound $A$. Then for all $ x\in \R^n$, the continuous frame $(x_t)_{t\in\Omega}$ does $A^{-1/2}$-stable phase retrieval near $x $.
\end{thm}

\begin{proof}

Let   $x\in \R^n $ and $\vp>0$.  For all $\alpha> 0$ we denote $\Omega_{\alpha}= \{t \in \Omega: |\langle x,x_t\rangle|\geq \alpha \|x_t\| \}$.   Let $\Theta$ be the analysis operator of $(x_t)_{t \in \Omega}$ and let $\Theta_{\Omega^c_{\alpha}}$ be the analysis operator of $(x_t)_{t \in \Omega^c_{\alpha}}$. As $\Omega = \cup_{\alpha> 0}\Omega_{\alpha} $ there exists $\beta > 0$ so that if $0 < \alpha \leq \beta $ then $\|\Theta_{\Omega^c_{\beta}}\|^2< \epsilon $.
Let  $ y\in \R^n$ with $\|x-y\|< \beta$. As in the proof of proposition 2.1 we have that  
 $sign(\langle x,x_t\rangle) =sign(\langle y,x_t\rangle)  $ for all $t \in \Omega_{\beta} $. Then,

\begin{align*}
\big\| |\Theta(x) |- |\Theta(y) | \big\|_{L_2(\Omega)} ^2  &=   \int_{\Omega}    \big||\langle x ,x_t\rangle -|\langle y ,x_t\rangle | \big|^2  d\mu.\\
    &\geq   \int_{{\Omega}_{\beta}}    \big||\langle x ,x_t\rangle -|\langle y ,x_t\rangle | \big|^2  d\mu\\ 
&= \int_{{\Omega}_{\beta}}    |\langle x-y ,x_t\rangle |^2  d\mu \hspace{.5cm}\textrm{ as }\textrm{$sign(\langle x,x_t\rangle) =sign(\langle y,x_t\rangle) $ for $t\in \Omega_\beta$,}\\
&= \int_{\Omega}    |\langle x-y ,x_t\rangle |^2  d\mu - \int_{{\Omega}^c_{\beta}}    |\langle x-y ,x_t\rangle |^2  d\mu\\
&= \|\Theta(x-y)\|^2-\|\Theta_{\Omega^c_{\beta}}(x-y)\|^2\\
&\geq A\|x-y\|^2-\|\Theta_{\Omega^c_{\beta}}\|^2\|x-y\|^2\\
&> (A-\epsilon)\|x-y\|^2
\end{align*}
 Thus, 
 \begin{align*}
  \\ \\ A^{-1/2} \lim_{y\rightarrow x}\frac{ \| |T(x)|-|T(y) | \|}{\min_{|\lambda |=1} \| x-\lambda y\|} \geq 1  
 \end{align*}
This prove that $(x_t)_{t\in \Omega}$ does $A^{-1/2}$-stable phase retrieval near $x$.

\end{proof}

\section{Stable phase retrieval near $x\in\C^n$}\label{S:C}

We have that a continuous frame for $\R^n$ does stable phase retrieval near $x$ for every $x\in\R^n$.  However, the situation is completely different for $\C^n$.  Indeed, although an ortho-normal basis for $\R^n$  does $1$-stable phase retrieval near $x\in\R^n$ for every $x\in\R^n$, we will show that the opposite holds for $\C^n$.  That is, if $(e_j)_{j=1}^n$ is a basis for  $\C^n$ and $x\in\C^n$ is not a scalar multiple of $e_j$ for some $1\leq j\leq n$ then $(e_j)_{j=1}^n$ does not do stable phase retrieval near $x$.

In order to prove that a continuous frame $(x_t)_{t\in\Omega}$ does not do stable phase retrieval near a vector $x\in\C^n$, we will find a vector $y\in\C^n$ and prove that 
\begin{equation}\label{E:0}
\lim_{\alpha\rightarrow 0} \frac{\||\Theta x|-|\Theta (x+\alpha y)|\|}{\min_{|\lambda|=1}\|x-\lambda(x+\alpha y)\|}=0.
\end{equation}

Proving \eqref{E:0} will involve showing that the numerator converges to $0$ at a faster rate than the denominator.  We first show that the denominator always converges at a linear rate.

\begin{lem}\label{L:denom}
Let $H$ be a Hilbert space and let $x,y\in H$ be linearly independent vectors.  Then there exists a constant $c>0$ so that for all scalars $\alpha\in\R$,
$$\min_{|\lambda|=1}\|x-\lambda(x+\alpha y)\|\geq |\alpha|c.
$$
\end{lem}
\begin{proof}
Let $P_{(span(x))^\perp}$ be orthogonal projection onto the orthogonal complement of the span of $x$.  We have that
$$\|x-\lambda(x+\alpha y)\|\geq \|P_{(span(x))^\perp}(x-\lambda(x+\alpha y))\|=|\alpha |\|P_{(span(x))^\perp}(y)\|$$
As $x$ and $y$ are linearly independent, we have that $\|P_{(span(x))^\perp}(y)\|\neq0$.
\end{proof}

\begin{thm}\label{T:onb}
Let $(e_j)_{j=1}^n$ be a basis of $\C^n$.  If $x\in\C^n$ is not a scalar multiple of $e_j$ for all $1\leq j\leq n$ then $(e_j)_{j=1}^n$ does not do stable phase retrieval near $x$.
\end{thm}

\begin{proof}
Let $x\in\C^n$ such that $x$ is not a scalar multiple of $e_j$ for all $1\leq j\leq n$ and let $(f_j)_{j=1}^n$ be bi-orthogonal to $(e_j)_{j=1}^n$. Without loss of generality, we may assume that $\langle x,e_1\rangle\neq0$ and $\langle x,e_2\rangle\neq0$. The basis expansion of $x$ with respect to $(f_j)_{j=1}^n$ is given by  $x=\sum_{j=1}^n\langle x,e_j\rangle f_j$. Let $y=\langle x,e_1\rangle if_1-\langle x,e_2\rangle if_2$.  Let $\Theta $ be the analysis operator of $(e_j)_{j=1}^n$.  We will prove that $$
\lim_{\alpha\rightarrow 0} \frac{\| |\Theta(x)|-|\Theta(x+\alpha y)|\|}{\min_{|\lambda|=1}\|x-\lambda(x+\alpha y)\|}=0
$$
 We have that $\langle y, e_1\rangle=\langle x,e_1\rangle i$ and $\langle y, e_2\rangle=-\langle x,e_2\rangle i$.  Thus, $x$ and $y$ are linearly independent. By Lemma \ref{L:denom} there exists a constant $c>0$ so that
$$\min_{|\lambda|=1}\|x-\lambda(x+\alpha y)\|\geq |\alpha|c.\hspace{1cm}\textrm{ for all }\alpha\in\R.
$$ 
We will show that there exists $k>0$ so that $\| |\Theta(x)|-|\Theta(x+\alpha y)|\| \leq k\alpha^2$ for all $\alpha\in\R$.  By Taylor's Approximation Theorem, $(1+\alpha^2)^{1/2}\leq 1+ \frac{1}{2}\alpha^2$ for all $\alpha\in\R$.  We now obtain the following upper bound.
\begin{align*}
\big\| |\Theta(x)|-|\Theta(x+\alpha y)|\big\|^2&= \big||\langle x,e_1\rangle| - |\langle x,e_1\rangle+\alpha\langle x,e_1\rangle i|\big|^2+\big||\langle x,e_2\rangle| - |\langle x,e_2\rangle-\alpha\langle x,e_2\rangle i|\big|^2\\
&=(|\langle x,e_1\rangle|^2+|\langle x,e_2\rangle|^2)\big(1-(1+\alpha^2)^{1/2}\big)^2\\
&\leq (|\langle x,e_1\rangle|^2+|\langle x,e_2\rangle|^2) \frac{1}{4}\alpha^4
\end{align*}

Thus, we have that  
$$
\lim_{\alpha\rightarrow 0} \frac{\| |\Theta(x)|-|\Theta(x+\alpha y)|\|}{\min_{|\lambda|=1}\|x-\lambda(x+\alpha y)\|}\leq \lim_{\alpha\rightarrow 0} \frac{(|\langle x,e_1\rangle|^2+|\langle x,e_2\rangle|^2)^\frac{1}{2} \frac{1}{2} \alpha^2}{c|\alpha|} = 0 
$$

\end{proof}

We extend the ideas in the previous proof to prove the following theorem.  It is noted in \cite{CPT} in the discussion after Example 3 that if there exists orthogonal unit vectors $x,y\in\C^n$ such that the frame coefficients of both $x$ and $y$ are real then the frame does not do phase retrieval.  We now prove that this is true locally as well.

\begin{thm}\label{T:frameCn}
Let $n\geq 2$ and let  $(x_j)_{j\in J}$ be a finite frame of $\C^n$.
Suppose that $x,y\in \C^n$ are linearly independent and that $\langle x,x_j\rangle$ and $\langle y,x_j\rangle$ are both real for all $j\in J$.  Then there exists $z\in span\{x,y\}$ such that $(x_j)_{j\in J}$ does not do stable phase retrieval near $z$.
\end{thm}

\begin{proof}
We may choose a non-zero real number $a\in\R$ so that for $z=ax+y$ we have for all $j\in J$  that if $\langle y,x_j\rangle\neq 0$ then
$\langle z,x_j\rangle\neq 0$.  Let $I=\{j\in J: \langle y,x_j\rangle\neq 0\}$.

 Let $\Theta $ be the analysis operator of $(x_j)_{j\in J}$.  We will prove that $$
\lim_{\alpha\rightarrow 0} \frac{\| |\Theta(z)|-|\Theta(z+\alpha y)|\|}{\min_{|\lambda|=1}\|z-\lambda(z+\alpha y)\|}=0
$$
 We have that $z$ and $y$ are linearly independent. By Lemma \ref{L:denom} there exists a constant $c>0$ so that for all $\alpha\in\R$,
$$\min_{|\lambda|=1}\|z-\lambda(z+\alpha y)\|\geq |\alpha|c.
$$ 
We will show that there exists $k>0$ so that $\| |\Theta(z)|-|\Theta(z+\alpha y)|\| \leq k\alpha^2$ for all $\alpha\in\R$.  We now obtain the following upper bound.
\begin{align*}
\big\| |\Theta(z)|-|\Theta(z+\alpha y)|\big\|^2&=\sum_{j\in I} \big||\langle z,x_j\rangle| - |\langle z,x_j\rangle+\alpha\langle y,x_j\rangle i|\big|^2\\
&=\sum_{j\in I}|\langle z,x_j\rangle|^2\big(1-(1+\beta^2)^{1/2}\big)^2\hspace{.5cm}\textrm{ for $\beta_j=\langle z,x_j\rangle^{-1}\langle y,x_j\rangle\alpha$}\\
&\leq \sum_{j\in I}|\langle z,x_j\rangle|^2 \frac{1}{4}\beta_j^4\hspace{1cm}\textrm{ by Taylor's Approximation Theorem,}\\
&=\sum_{j\in I}|\langle z,x_j\rangle|^{-2}|\langle y,x_j\rangle|^4 \frac{1}{4}\alpha^4
\end{align*}

Thus, for $k=2^{-1}(\sum_{j\in I}|\langle z,x_j\rangle|^{-2}|\langle y,x_j\rangle|^4)^{1/2}$ we have that  
$$
\lim_{\alpha\rightarrow 0} \frac{\| |\Theta(z)|-|\Theta(z+\alpha y)|\|}{\min_{|\lambda|=1}\|z-\lambda(z+\alpha y)\|}\leq \lim_{\alpha\rightarrow 0} \frac{k\alpha^2}{c|\alpha|} = 0 
$$

\end{proof}

\section{Stable phase retrieval near $x$ in infinite dimensions}\label{S:inf}

We previously proved that if $(x_j)_{j\in J}$ is a frame or a basis of $\R^n$ with lower frame bound $A$ then for all $x\in\R^n$ we have that $(x_j)_{j\in J}$ does $A^{-1/2}$-stable phase retrieval.  We now consider the situation for infinite dimensional spaces.

\begin{prop}\label{P:basis}
Let $(x_j)_{j=1}^\infty$ be a Riesz basis for a real infinite dimensional Hilbert space $H$.  Let $x\in H$.  Then $(x_j)_{j=1}^\infty$ does stable phase retrieval near $x$ if and only if $\langle x,x_j\rangle=0$ for all but finitely many $j\in\N$.
\end{prop}
\begin{proof}
Let $x\in H$ and $J=\{j\in\N:\langle x,x_j\rangle\neq0\}$.  We first assume that $J$ is a finite set. There exists $\vp$ so that if $y\in H$ and $\|x-y\|< \vp$ then  $sign(\langle x ,x_j\rangle)=sign(\langle x ,x_j\rangle)$ for all $j\in J$. 
We now assume that $y\in H$ and $\|x-y\|< \vp$.  Let $A$ be the lower frame bound for the Riesz basis $(x_j)_{j=1}^\infty$ Then
\begin{align*}
\sum_{j=1}^\infty&\big||\langle x,x_j\rangle|-|\langle y,x_j\rangle|\big|^2=\sum_{j\in J}\big||\langle x,x_j\rangle|-|\langle y,x_j\rangle|\big|^2 + \sum_{j \notin J}|\langle y,x_j\rangle|^2 \hspace{.5cm}\textrm{ as $\langle x,x_j\rangle=0$ for $j\not\in J$},\\
  &=\sum_{j\in J}|\langle x,x_j\rangle-\langle y,x_j\rangle|^2 + \sum_{j \notin J}|\langle y,x_j\rangle|^2\hspace{.5cm}\textrm{ as }\textrm{$sign(\langle x,x_j\rangle) =sign(\langle y,x_j\rangle) $ for $j\in J$,}\\
  &=\sum_{j\in J}|\langle x-y,x_j\rangle|^2 + \sum_{j \notin J}|\langle x-y,x_j\rangle|^2\hspace{.5cm}\textrm{ as $\langle x,x_j\rangle=0$ for $j\not\in J$},\\
  &=\sum_{j=1}^\infty|\langle x,x_j\rangle-\langle y,x_j\rangle|^2 \geq A \|x-y\|^2\hspace{.5cm}\textrm{ as $(x_j)_{j=1}^\infty$ has lower frame bound $A$.}
\end{align*}
So, $(x_j)_{j=1}^\infty$ does $A^{-\frac{1}{2}}$- stable phase retrieval at $x$. 

We now assume  that $J$ is infinite. Let $\vp >0$. Let $(y_j)_{j=1}^\infty$ be the biorthogonal sequence to $(x_j)_{j=1}^\infty$. That is $\langle y_j ,x_j\rangle = 1$ and  $\langle y_j ,x_i\rangle = 0$ for $i\neq j$. As $(x_j)_{j=1}^\infty$ is Riesz basis we have that $(y_j)_{j=1}^\infty$ is a Riesz basis as well and there exists $c>0$ so that $\|y_j\|\leq c$ for all $j\in \N$. As $\sum|\langle x ,x_{j_0}\rangle|^2$ converges there exits $j_0\in J$ so that $|\langle x, x_{j_0}\rangle| <\vp c^{-1}$. Let $y=x-2\langle x , x_{j_0} \rangle y_{j_0}$.   Thus, 
$$\|x-y\|=2|\langle x, x_{j_0}\rangle|\|  y_{j_0}\|< 2 \vp c^{-1} c=2 \vp$$

We have that  $x\neq y $ and $x\neq - y $ as $\langle x,x_{j_0}\rangle\neq0$. Furthermore, 
$\langle x, x_{j_0}\rangle = -\langle y, x_{j_0}\rangle$ and 
$\langle x, x_{j}\rangle = \langle y, x_{j}\rangle$ for all $j \neq j_0$.  Thus, if $\Theta$ is the analysis operator of $(x_j)_{j=1}^\infty$, we have that $|\Theta x|=|\Theta y|$ but that $x\neq \lambda y$ for all $|\lambda|=1$.  Hence, $(x_j)_{j=1}^\infty$ does not do phase retrieval near $x$.

\end{proof}

Cahill, Casazza, and Daubechies proved that no frame for an infinite dimensional Hilbert space does stable phase retrieval \cite{CCD}, which was generalized to continuous frames for infinite dimensional Hilbert spaces by Alaifari and Grohs \cite{AG} (However, it is proven in \cite{CDFF} that stable phase retrieval is possible for infinite dimensional subspaces of $L_2(\R)$).  In both proofs of instability, the authors show that if $\vp>0$, $(x_t)_{t\in\Omega}$ is a continuous frame of an infinite dimensional Hilbert space $H$, and $f\in H$ has $\|f\|=1$ then it is possible to choose $g\in H$ with $\|g\|=1$, $\langle f,g\rangle=0$, and $\|\min(|\Theta f|,|\Theta g|)\|_{L_2(\Omega)}<\vp$.  It then follows for $x=f+g$ and $y=f-g$ that $\|x-\lambda y\|=2$ for all $|\lambda|=1$ but that $\||\Theta x|-|\Theta y|\|_{L_2(\Omega)}<2\vp$.  Note that the vectors $x$ and $y$ which witness the instability are orthogonal and as $\vp\searrow0$ the corresponding nets $(x_\vp)$ and $(y_\vp)$ are both not convergent.  We know by Theorem \ref{T:ortho} that instability of phase retrieval is always greatest at orthogonal vectors, and by Proposition \ref{P:basis} we know that there exists examples of continuous frames for infinite dimensional Hilbert spaces which do stable phase retrieval near some vectors.  In the following theorem we extend the results of \cite{CCD} and \cite{AG} by proving that the set of vectors where a continuous frame for an infinite dimensional Hilbert space $H$ does not do stable phase retrieval is dense in $H$.

\begin{thm}
Let $(x_t)_{t\in\Omega}$ be a continuous frame for an infinite dimensional Hilbert space $H$.  Then there exists $x\in H$ such that $(x_t)_{t\in\Omega}$ does not do stable phase retrieval near $x$.  Furthermore, the set of $x\in H$ where $(x_t)_{t\in\Omega}$ does not do stable phase retrieval near $x$ is dense in $H$.
\end{thm}
\begin{proof}
Let $(x_t)_{t\in\Omega}$ be a continuous frame for an infinite dimensional Hilbert space $H$ over the measure space $(\Omega,\mu)$ and let $\Theta:H\rightarrow L_2(\Omega)$ be the analysis operator.  Without loss of generality we may assume that $x_t\neq 0$ for all $t\in\Omega$.  By replacing $(x_t)_{t\in\Omega}$ 
with $(x_t/\|x_t\|)_{t\in\Omega}$ and replacing the measure $d\mu$ with $\|x_t\|^2 d\mu$ we are able to assume that $\|x_t\|=1$ for all $t\in\Omega$. We are able to make this change of measure
because for all measureable $J\subseteq\Omega$ we have that
$$\int_J |\langle x,x_t\rangle|^2 d\mu(t)=\int_J \big|\big\langle x,\frac{1}{\|x_t\|}x_t \big\rangle\big|^2 \|x_t\|^2 d\mu(t)\hspace{1cm}\textrm{ for all }x\in H.
$$

For simplicity, we assume that $(x_t)_{t\in\Omega}$ has a lower frame bound $A>1$.
  For a measurable set $\Omega'\subseteq\Omega$, we define the analysis operator restricted to $\Omega'$ by  
$\Theta_{\Omega'}(x)=(1_{\Omega'}(t)\langle x,x_t\rangle)_{t\in\Omega}$.  Our assumption that $\|x_t\|=1$ for all $t\in\Omega$ implies that $\Theta_{\Omega'}$ is a compact operator whenever $\Omega'$ has finite measure.  
 Let $z_0\in H$ with $\|z_0\|=1$. 
Let $\vp_j\searrow0$ such that $\vp_j<4^{-j}$ for all $j\in\N$. We now claim that we may build an ortho-normal sequence of vectors $(z_j)_{j=0}^\infty \subseteq H$ and a pairwise disjoint sequence of finite measure sets $(\Omega_j)_{j=0}^\infty\subseteq\Omega$ such that 
\begin{enumerate}
    \item $\|\Theta_{\Omega_j} z_j\|\geq 1$\hspace{1cm} for all $j\in\N_0$,\\
    \item $\|\Theta_{\Omega^c_j} z_j\|\leq 4^{-j}$\hspace{1cm} for all $j\in\N_0$,\\
        \item $\|\Theta_{\Omega_i} z_j\|\leq 2^{-1}4^{-i}$\hspace{1cm} for all $j\neq i$.
\end{enumerate}

We now prove the claim by induction. We have already fixed $z_0\in H$, and as $(x_t)_{t\in\Omega}$ has lower frame bound $A>1$ we may choose a finite  measure subset $\Omega_0\subset \Omega$ so that $\|\Theta_{\Omega_0} z_0\|\geq 1$ and $\|\Theta_{\Omega^c_0} z_0\|\leq 1$.  Thus, (1) and (2) are satisfied and we have that (3) is vacuously true. 

We now let $k\in\N_0$ and assume that $(z_j)_{j=0}^k$ and   $(\Omega_j)_{j=0}^k$ have been chosen.  Choose a finite measure subset $\Omega'\subseteq \Omega$ so that $\cup_{j=0}^k \Omega_j\subseteq \Omega'$ and $\|\Theta_{\Omega'^c} z_j\|\leq 4^{-k-1}$ for all $0\leq j\leq k$.  
 Thus, (3) is satisfied for $0\leq j\leq k$ and $i=k+1$ as long as $\Omega_{k+1}\subseteq \Omega'^c$.
Let $\vp>0$ so that $\vp<A-1$ and $\vp<4^{-k-1}$.  As $\Theta_{\Omega'}$ is a compact operator we may choose a normalized vector $z_{k+1}$ in $H$ such that $z_j$ is orthogonal to $(z_j)_{j=0}^k$ and $\|\Theta_{\Omega'} z_{k+1}\|< \vp$.  As $\vp<4^{-k-1}$, we have that (3) is satisfied for  $j=k+1$ and $0\leq i\leq k$.  As $\vp<A-1$ we may choose a finite measure subset $\Omega_{k+1}\subseteq \Omega'^c$ so that $\|\Theta_{\Omega_{k+1}} z_{k+1}\|\geq 1$ and $\|\Theta_{\Omega^c_{k+1}} z_{k+1}\|\leq 4^{-k-1}$.  Thus, (1) and (2) are satisfied and our induction is complete.

As a consequence of (3) and $(x_t)_{t\in\Omega}$ having lower frame bound $A>1$, we have for all $k\geq 2$ that,

\begin{equation}\label{E:lower1}
\Big\|\Theta_{\Omega_k^c}\sum_{j\neq k} 2^{-j}z_j\Big\|\geq \Big\|\Theta \sum_{j\neq k} 2^{-j}z_j\Big\|-\sum_{j\neq k}2^{-j}\|\Theta_{\Omega_k}z_j \|\geq (\sum_{j\neq k}2^{-2j})^{1/2}-4^{-k}> (1+4^{-1})^2- 4^{-2}> 1
\end{equation}

We let $x=\sum_{j=0}^\infty 2^{-j}z_j$.  Let  $k\geq 2$,  $y=-2^{-k}z_k+\sum_{j\neq k}2^{-j}z_j$, and let $\lambda$ be a scalar with $|\lambda|=1$.  We have the following lower bound for $B\|x-\lambda y\|^2$ where $B$ is the upper frame bound of $(x_t)_{t\in\Omega}$.

\begin{align*}
    &B\|x-\lambda y\|^2\geq \|\Theta (x-\lambda y)\|^2\\
    &= \|\Theta_{\Omega_k^c} (x-\lambda y)\|^2+\|\Theta_{\Omega_k} (x-\lambda y)\|^2\\
     &= \big\|(1-\lambda)\Theta_{\Omega_k^c}(\sum_{j\neq k}2^{-j}z_j)+(1+\lambda)\Theta_{\Omega_k^c}(2^{-k}z_k) \big\|^2+\big\|(1+\lambda)\Theta_{\Omega_k}2^{-k}z_k)+ (1-\lambda)\Theta_{\Omega_k}(\sum_{j\neq k}2^{-j}z_j)\big\|^2\\
          &\geq \Big|\big\|(1-\lambda)\Theta_{\Omega_k^c}\sum_{j\neq k}2^{-j}z_j\big\|-\big\|(1+\lambda)2^{-k}\Theta_{\Omega_k^c}z_k \big\|\Big|^2+\Big|\big\|(1+\lambda)2^{-k}\Theta_{\Omega_k}z_k\big\|-\big\| (1-\lambda)\Theta_{\Omega_k}\sum_{j\neq k}2^{-j}z_j\big\|\Big|^2\\
&\geq \frac{|1-\lambda|^2}{2}\big\|\Theta_{\Omega_k^c}\sum_{j\neq k}2^{-j}z_j\big\|^2-4\cdot2^{-2k}\|\Theta_{\Omega_k^c}z_k \|^2+\frac{|1+\lambda|^2}{2}2^{-2k}\|\Theta_{\Omega_k}z_k\|^2- 4\big\|\Theta_{\Omega_k}\sum_{j\neq k}2^{-j}z_j\big\|^2\\
&\hspace{4cm}\textrm{ as $(a-b)^2\geq (1/2)a^2-b^2$ for all $a,b\in\R$,} \\
&\geq \frac{|1-\lambda|^2}{2}-4\cdot 2^{-2k}4^{-2k}+\frac{|1+\lambda|^2}{2} 2^{-2k}- 4\big\|\sum_{j\neq k}2^{-j}\Theta_{\Omega_k}z_j\big\|^2\hspace{.5cm}\textrm{ by \eqref{E:lower1}, (2), (1),}\\
&\geq (|1-\lambda|^2+|1+\lambda|^2) 2^{-2k-1}-4^{-3k+1} -4(\sum_{j\neq k}2^{-j}\|\Theta_{\Omega_k}z_j\|)^2\\
&\geq 2^{-2k-1}-4^{-3k+1} -4\cdot 4^{-2k} \hspace{.5cm}\textrm{ by  (3).}\\
\end{align*}

On the other hand, we have that 
\begin{align*}
    \big\| |\Theta x|&-|\Theta y| \big\|^2 =\big\| |\Theta_{\Omega_k^c} x|- |\Theta_{\Omega_k^c} y|\big\|^2+\big\| |\Theta_{\Omega_k} x|-|\Theta_{\Omega_k}y|\big\|^2\\
    &= \Big\| \big|\Theta_{\Omega_k^c}(2^{-k} z_k+\sum_{j\neq k}2^{-j}z_j)\big|- \big|\Theta_{\Omega_k^c}(-2^{-k} z_k+\sum_{j\neq k}2^{-j} z_j)\big|\Big\|^2\\
    &\hspace{1cm}+\Big\| \big|\Theta_{\Omega_k} (2^{-k} z_k+\sum_{j\neq k}2^{-j}z_j)\big|-\big|\Theta_{\Omega_k}(-2^{-k} z_k+\sum_{j\neq k}2^{-j} z_j)\big|\Big\|^2\\
    &\leq  \Big(2^{-k} \| \Theta_{\Omega_k^c}z_k\| + 2^{-k} \| \Theta_{\Omega_k^c}z_k\|\Big)^2+ \Big(\big\| \Theta_{\Omega_k} \sum_{j\neq k}2^{-j}z_j\big\|+\big\|\Theta_{\Omega_k}\sum_{j\neq k}2^{-j} z_j\big\|\Big)^2\\
        &\leq  2^{-2k+2}\| \Theta_{\Omega_k^c}z_k\|^2+ 4 \Big(\sum_{j\neq k}2^{-j}\|\Theta_{\Omega_k} z_j\|\Big)^2\\
        &\leq  2^{-2k+2} 4^{-2k} + 4\Big(\sum_{j\neq k}2^{-j-1}4^{-k}\Big)^2\hspace{.5cm}\textrm{ by  (2) and (3),}\\
                &\leq  4^{-3k+1} + 4^{-2k+1}\\
\end{align*}

Thus, we have that $B\min_{|\lambda|=1}\|x-\lambda y\|^2\geq  2^{-2k-1}-4^{-3k+1} -4^{-2k+1}$ and $\| |\Theta x|-|\Theta y| \|^2\leq 4^{-3k+1} + 4^{-2k+1}$.  Hence, $ \lim_{y\rightarrow x} \| |\Theta x|-|\Theta y| \|/\min_{|\lambda|=1}\|x-\lambda y\|=0$.  This proves that $(x_t)_{t\in\Omega}$ does not do stable phase retrieval near $x$.  

Note that the same proof would give for all $N\in\N$ that $(x_t)_{t\in\Omega}$ does not do stable phase retrieval at $z_0+\sum_{j=N}^\infty 2^{-j}z_j$.  As, $z_0=\lim_{N\rightarrow\infty} z_0+\sum_{j=N}^\infty 2^{-j}z_j$ we have that the set of vectors $x$ in $H$ where $(x_t)_{t\in\Omega}$ does not do stable phase retrieval at $x$ is dense in $H$.

\end{proof}

\end{document}